\documentclass[11pt,reqno]{amsart}
\usepackage{amsfonts}
\usepackage{amssymb,latexsym}
\usepackage{enumerate}
\usepackage{textcomp}
\usepackage{mathrsfs}
\usepackage{amsmath}
\usepackage{amsthm}
\usepackage{hyperref}
\usepackage[square, comma, sort&compress, numbers]{natbib}
\usepackage{bm}
\usepackage{fancyhdr}
\usepackage{lastpage}
\usepackage{layout}
\usepackage{ulem}
\usepackage{esint}
\usepackage{extarrows}
\usepackage{ulem}
\usepackage[marginal]{footmisc}
\newtheorem{theorem}{Theorem}[section]

\newtheorem{lemma}[theorem]{Lemma}
\newtheorem{proposition}[theorem]{Proposition}
\newtheorem{definition}[theorem]{Definition}
\newtheorem{conjecture}[theorem]{Conjecture}

\newtheorem{problem}[theorem]{Problem}
\newtheorem{remark}[theorem]{\bf{Remark}}

\numberwithin{equation}{section}

\newcommand{\bp}{\bar{\partial}}
\newcommand{\beq}{\begin{equation}}
	\newcommand{\eeq}{\end{equation}}
\newcommand{\beqn}{\begin{equation*}}
	\newcommand{\eeqn}{\end{equation*}}
\newcommand{\C}{\mathbb{C}}
\newcommand{\R}{\mathbb{R}}
\newcommand{\N}{\mathbb{N}}
\newcommand{\CP}{\mathbb{P}}
\newcommand{\T}{\mathbb{T}}
\newcommand{\supp}{\mathrm{supp}}
\newcommand{\pp}{\partial\bar{\partial}}
\newcommand{\p}{\partial}

\newcommand{\w}{\omega}
\newcommand{\ka}{\kappa}
\newcommand{\im}{\sqrt{-1}}

\newcommand{\mF}{\mathcal{F}}
\newcommand{\mO}{\mathcal{O}}

\DeclareMathOperator \rank{rank}
\DeclareMathOperator \tr{tr}

\DeclareMathOperator{\Rm}{Rm}
\DeclareMathOperator{\rc}{Ric}
\DeclareMathOperator{\spaned}{span}

\DeclareMathOperator{\duallefs}{\Lambda}

\begin{document}
	
	\vspace*{0.5cm} \begin{center}\noindent {\LARGE \bf On the structure of compact K\"ahler manifolds with nonnegative holomorphic sectional curvature  }\\[0.7cm]
		Shiyu Zhang\footnotemark[1]
		\\[0.2cm]
		School of Mathematical Sciences\\
		University of Science and Technology of China\\ Hefei, 230026, P.R. China\\
		E-mail: shiyu123@mail.ustc.edu.cn\\[0.5cm]
		Xi Zhang\footnotemark[1]
		\\[0.2cm]
		School of Mathematics and Statistics\\
		Nanjing University of Science and Technology\\
		Nanjing, 210094, P.R. China\\
		E-mail: mathzx@njust.edu.cn\\[1cm]
		\footnotetext[1]{\notag\noindent The research was supported by the National Key R and D Program of China 2020YFA0713100, All authors were supported in part by NSF in China, No.12141104, 12371062 and 11721101.}
		
	\end{center}
	

	{\bf Abstract.} In this paper, we establish a ``pseudo-effective" version of the holonomy principle for compact K\"ahler manifolds with nonnegative holomorphic sectional curvature. As applications, we prove that if a compact complex manifold $M$ admits a K\"ahler metric $\omega$ with nonnegative holomorphic sectional curvature and $(M,\w)$ has no nonzero truly flat tangent vector at some point (which is satisfied when the holomorphic sectional curvature is quasi-positive), then $M$ must be projective and rationally connected. This result answers a problem raised by Yang and Matsumura and extends the Yau's conjecture. We also prove that a compact simply connected K\"ahler manifold with nonnegative holomorphic sectional curvature is projective and rationally connected. Additionally, we classify non-projective compact K\"ahler 3-dimensional manifolds. Furthermore, we show that a compact K\"{a}hler manifold admits a Hermitian metric with positive real bisectional curvature is a projective and rationally connected manifold.
	\\[1cm]
	{\bf AMS Mathematics Subject Classification.} 53C07, 58E15 \\
	{\bf
		Keywords and phrases.} K\"ahler metric,\ holomorphic sectional curvature,\ projective,\ rational connectedness
	
	
	\newpage
	
	\section{Introduction}
	
	A projective manifold $M$ is called rationally connected if any two points on $M$ can be connected by some rational curve. The rational connectedness is an important and fundamental concept in algebraic geometry, and many people have given the criteria for it (\cite{campana2015,GHS,Pet,LP17,Cam16,KMM}). It is also important to provide geometric interpretations for rational connectedness. Let $\omega $ be a K\"ahler metric on $M$, the holomorphic sectional curvature of $\omega$ along a nonzero  $V\in T^{1,0}M$ is defined by
	$$H(V)=\frac{R(V,\overline{V},V,\overline{V})}{\vert V\vert_{\omega}^4},$$
	where $R$ is the Riemannian curvature tensor of $\omega$.
	S.-T. Yau proposed the following well-known conjecture in his problem section:
	\begin{conjecture}[{\cite[Problem 47]{yau1982}}]\label{yau}
		If a compact complex manifold $M$ admits a K\"{a}hler metric $\omega$ with positive holomorphic sectional curvature, then $M$ is projective and rationally connected.
	\end{conjecture}
	Heier and Wong (\cite{heier2020}) confirmed Yau's conjecture for projective manifolds with quasi-positive holomorphic sectional curvature by studying MRC fibration.
	By introducing the notion of RC-positivity and studying vanishing theorems, Yang (\cite{yang2018,yang2020}) showed that a K\"ahler manifold $M$ with positive holomorphic sectional curvature is projective and rationally connected, then affirmly confirmed Yau's conjecture.
	
	There are some other interesting curvature conditions related to rational connectedness (see, for instance, \cite{ni2021,ni2022,yang2018,yang2020}). Recently, Li, Zhang and the second author (\cite[Theorem 1.14]{li2021}) established the equivalence between mean curvature quasi-positivity of the tangent bundle and rational connectedness for compact K\"{a}hler manifolds.

	In \cite{matsumura2020,matsumura2022},  Matsumura established a structure theorem for projective manifolds $X$ with nonnegative holomorphic
	sectional curvature by studying the structure of MRC fibration. Both the methods of Heier-Wong and Matsumura rely on the uniruledness criterion for projective manifolds (\cite{BDPP}). According to Matsumura's structure theorem, the universal cover of $X$ splits isometrically and holomorphically as $F\times \C^{n-k}$, where $F$ is a rationally connected manifold. This result generalizes the structure theorem established by Howard, Smith and Wu (\cite{HSW}) and Mok (\cite{mok}) for nonnegative holomorphic bisectional curvature. As a corollary, Matsumura generalized Yau's conjecture for projective manifolds (\cite[Theorem 1.2]{matsumura2022}): if a projective manifold $M$ admits a K\"{a}hler metric $\w$ with nonnegative holomorphic sectional curvature and $(M,\w)$ has no nonzero truly flat tangent vector at some point (which is satisfied when the holomorphic sectional curvature is quasi-positive), then $M$ is rationally connected.  According to \cite{heier2018}, a tangent vector $V\in T_{x_0}^{1,0}M$ is said to be truly flat at $x_0$ with respect to $\w$ if
	$$R(V,\overline{X},Y,\overline{Z})=0,\text{\ for all $X,Y,Z\in T^{1,0}M$}.$$
	Motivated by these results, Matsumura and Yang proposed the following natural generalization.
	
	\begin{conjecture}[{\cite[Conjecture 1.9]{yang2020}}, {\cite[Problem 3.3]{matsumuraopen}}]\label{c}
		Let $M$ be a compact K\"{a}hler manifold. If $M$ admits a K\"ahler metric (or more generally a Hermitian metric) with nonnegative holomorphic sectional curvature and there has no nonzero truly flat tangent vector at some point, then $M$ is projective and rationally connected.
	\end{conjecture}
	
	More generally, Matsumura proposed the following extension and suggested investigating Albanese maps instead of MRC fibrations for compact K\"ahler manifolds, due to the lack of a uniruledness criterion in the K\"ahler case. We refer the reader to \cite[Section 3]{matsumuraopen} and \cite[Section 5]{matsumura2022} for a more detailed discussion.
	
	\begin{problem}[{\cite[Problem 3.5]{matsumuraopen}}]\label{stru}
		Can the structure theorem for projective manifolds with nonnegative holomorphic sectional curvature be generalized to compact K\"ahler manifolds?
	\end{problem}
	
	Towards Conjecture \ref{c} and Problem \ref{stru}, we establish a	``pseudo-effective" version of the holonomy principle for nonnegative holomorphic sectional curvature (c.f. \cite[Theorem 1.5]{campana2015} for nonnegative Ricci curvature). This is the first main result of our paper.
	
	\begin{theorem}\label{t2}
		Let $(M,\w)$ be a compact K\"ahler manifold with nonnegative holomorphic sectional curvature. Assume there exists an invertible subsheaf $\mathcal{F}\subset\mathcal{O}(\Lambda^{p,0}M)$ that is pseudo-effective as a line bundle. Then the following statements holds:
		\begin{itemize}
			\item[(1)] $\mathcal{F}$ is flat as a line bundle and invariant under parallel transport by the connection $D$ of $\Lambda^{p,0}M$ induced by Chern connection of $(T^{1,0}M,\w)$. In fact, the local holonomy $H$ of $(T^{1,0}M,\w)$ acts trivally on $\mF$.
			\item[(2)] $\mF$ determines a parallel and self-adjoint linear tranformation $S:T^{1,0}M\rightarrow T^{1,0}M$,
			which induces a decomposition $V+W$ of $T^{1,0}M$ where $V$ and $W$ consist of eigenvectors corresponding to nonzero and zero eigenvalues of $S$, respectively; in fact, $\forall p\in M$, every $v\in V_p$ is truly flat and $\dim V_p=\rank(S)>0$.
		\end{itemize}
	\end{theorem}
	
	The core of the proof of Theorem \ref{t2} is an integral inequality associated with $\mF$ (Proposition \ref{boch}). As a direct application of Theorem \ref{t2} and the Campana-Demailly-Peternell's criterion for rational connectedness (\ref{cdp}), we obtain the second main result of our paper, which confirms Conjecture \ref{c} for K\"ahler manifolds and extends Yau's conjecture to the quasi-positive case. Our approach is independent of those by Heier-Wong (\cite{heier2020}) and Matsumura (\cite{matsumura2022}) for projective manifolds.
	
	\begin{theorem}\label{quasi}
		Let $(M,\w)$ be a compact K\"{a}hler manifolds. If the holomorphic sectional curvature holomorphic sectional curvature is nonnegative and $(M,\w)$ has no nonzero truly flat tangent vector at some point, then $M$ is projective and rationally connected.
	\end{theorem}
	
	\begin{remark}
		Statement (1) of Theorem \ref{t2} parallels the statement of \cite[Theorem 1.5]{campana2015} for nonnegative Ricci curvature, but the proof is more involved. Based on this statement, Campana-Demailly-Peternell's criterion for rational connectedness (Lemma \ref{cdp}), Cheeger-Gromoll's splitting theorem (\cite{cg71}), De Rham's splitting theorem (\cite{dr52}) and Berger's classification (\cite{berger}) for irreducible holonomy, Campana, Demailly and Peternell established the structure theorem of a compact K\"{a}hler manifold $M$ with nonnegative Ricci curvature (\cite[Theorem 1.4]{campana2015}): the universal cover $\widetilde{M}$ of $M$ can splits isometrically as $\C^q\times\prod Y_j\times\prod S_k\times\prod Z_l$, where $Y_j$, $S_k$ and $Z_l$ are compact simply connected K\"{a}hler manifolds of respective dimensions $n_j$, $n'_k$, $n''_l$ with irreducible holonomy. Specifically, $Y_j$ are Calabi-Yau manifolds (holonomy $SU(n_j)$), $S_k$ are holomorphic symplectic manifolds (holonomy $Sp(n'_k / 2)$), and $Z_l$ are rationally connected manifolds with $K_{Z_l}^{-1}$ semipositive (holonomy $U(n''_l)$, unless $Z_l$ is Hermitian symmetric of compact type).
	\end{remark}
	
	Comparing with \cite[Theorem 1.4]{campana2015} regarding nonnegative Ricci curvature, Ricci-flat factors can be excluded for a compact K\"ahler manifold with nonnegative holomorphic sectional curvature. As another application of Theorem \ref{t2}, we affirm Problem \ref{stru} for the simply connected case, which constitutes the third main result of our paper.
	
	\begin{theorem}\label{sc}
		If $(M,\w)$ is a compact simply connected K\"ahler manifold with nonnegative holomorphic sectional curvature, then $M$ is projective and rationally connected.
	\end{theorem}
	
	According to Theorem \ref{sc}, Problem \ref{stru} is equivalent to the following splitting theorem for nonnegative holomorphic sectional curvature, which is an analogue of Cheeger-Gromoll's splitting theorem (\cite{cg71}) for nonnegative Ricci curvature.
	\begin{problem}\label{sp}
		Let $(M,\w)$ be a compact K\"ahler manifold with nonnegative holomorphic sectional curvature. Can the universal cover $\widetilde{M}$ of $M$ splits isometrically as $N\times \C^k$ for some compact K\"ahler manifold $N$?
	\end{problem}
	
	Matsumura has affirmed Problem \ref{stru} for compact K\"ahler surfaces (\cite[Corollary 1.5]{matsumura2020}). To support Problem \ref{stru} and \ref{sp} in dimension 3, we provide the following classification of non-projective K\"ahler 3-dimensional manifolds with nonnegative holomorphic sectional curvature by considering Albanese maps. This classification can also serve as a concrete example.
	\begin{theorem}\label{lowdim}
		Let $(M,\w)$ be a non-projective compact K\"ahler 3-dimensional manifold with nonnegative holomorphic sectional curvature. Then $M$ is either flat or a $\CP^1$-bundle over a torus $\T^2$. In the latter case, the universal cover of $M$ is $\CP^1\times \C^2$ endowed with the product of a K\"ahler metric with quasi-positive holomorphic sectional curvature on $\CP^1$ and the flat metric on $\C^2$.
	\end{theorem}
	\begin{remark}\label{re1}
		From Proposition \ref{cha} that was used in the proof of Theorem \ref{lowdim}, it is likely that there exists a parallel result similar to \cite[Proposition 3.10]{dps} for compact K\"{a}hler manifolds with nonnegative holomorphic sectional curvature. If this is the case, then Problem \ref{stru} could be resolved by considering Albanese maps as discussed in \cite[Section 3]{dps}. However, we currently do not have any information about the fundamental group of such manifolds. Therefore, from this perspective, it might be more necessary and direct to consider Problem \ref{sp}.
	\end{remark}

	Finally, we consider the Hermitian case. Yang (\cite[Corollary 1.10]{yang2018}) proved that a compact K\"{a}hler surface with a Hermitian metric of positive holomorphic sectional curvature is projective and rationally connected. It remains difficult to consider the Hermitian case in dimension $\geq$ 3. In \cite{YangZheng}, Yang and Zheng introduced the concept of real bisectional curvature and generalized Wu-Yau's Theorem (\cite{HLW16,wy16a,wy16b,dt19}) to the Hermitian case. This concept is also referred as "positive curvature operator" by Lee and Streets (\cite{ls2021}). Additionally, Yang and Zheng proposed a generalization of Yau's Conjecture \ref{yau} for the Hermitian case (\cite[Conjecture 1.6 (b)]{YangZheng}).
	
	\begin{conjecture}\label{c3}
		Let $(M,h)$ be a Hermitian manifold with positive real bisectional curvature. Then M is a projective and rationally connected manifold.
	\end{conjecture}
	
	We say that $(M,h)$ has positive real bisectional curvature denoted by $B_h>0$ if for any $x\in M$, any unitary frame $e=\{e_1,\cdots,e_n\}$ at $x$ and any nonnegative constants $a=\{a_1,\cdots,a_n\}$ with $\vert a\vert^2=a_1^2+\cdots+a_n^2>0$, it holds that $$B_h(e,a)=\frac{1}{\vert a\vert^2}\sum\limits_{i,j=1}^nR_{i\bar{i}j\bar{j}}a_ia_j>0,$$
	where $R$ is the Chern curvature tensor of $h$. This curvature condition is equivalent to positive holomorphic sectional curvature when the metric $h$ is K\"{a}hler, and it's slightly stronger in the non-K\"{a}hler case (cf. \cite{YangZheng}). Under the entra assumption that $M$ is K\"{a}hlerian, Conjcture \ref{c3} is just a weaker version of Conjecture \ref{c}. In this special case, Tang (\cite[Theorem 1.4]{tang2019}) confirmed the Conjecture \ref{c3} for $\dim M=3$ and proved that $M$ is projective for $\dim M\geq1$.  Applying the Bochner-type formula used in the proof of Proposition \ref{boch}, we confirm Conjecture \ref{c3} under the additional assumption that $M$ is K\"{a}hlerian.
	
	\begin{theorem}\label{t3}
		Let $(M,h)$ be a compact Hermitian manifold with positive real bisectional curvature. Then for any holomorphic vector bundle $E$ on $M$, there exists a constant $C_E>0$ such that
		$$H^0(M,((T^{1,0}M)^*)^{\otimes p}\otimes E^{\otimes l})=0\ \ \text{for all $p,l\in \N_+$ with $p\geq C_El$.}$$
		In particular, if $M$ is K\"{a}hlerian, $M$ is a projective and rationally connected manifold.
	\end{theorem}
	
	\medskip
	
	This paper is organized as follows. In Section \ref{pre}, we recall some basic definitions and facts. In Section \ref{s3}, we establish an integral inequality to show Theorem \ref{t2} and provides the proofs of Theorem \ref{quasi} and Theorem \ref{sc}. In section \ref{23}, we present a proof of Theorem \ref{lowdim}. In Section \ref{rbc}, we give a proof of Theorem \ref{t3}. In the appendix, we provide a proof of $L^2$-vanishing theorem.

	\section{Preliminaries}\label{pre}
	For the reader's convenience, we gather some basic definitions and useful facts relevant to the study of rational connectedness and holomorphic sectional curvature, which are well-known to experts and will be referenced later in the text

	\subsection{Pseudo-effective line bundle and plurisubharmonic function} We first introduce the concept of pseudo-effective line bundle (see e.g. \cite{de93,de10}).
	\begin{definition}
		A singular Hermitian metric $h$ on a line bundle $L$ is a metric which is given in any trivialization $\theta: F|_{\Omega}\rightarrow\Omega\times \C$ by
		$$\|\xi\|_h=|\theta(\xi)|e^{-\varphi(x)},\ x\in\Omega,\xi\in L_x$$
		where $\varphi\in L^1_{loc}(\Omega)$, called the weight of the metric with respect to the trivialization $\theta$.
	\end{definition}
	\begin{definition}
		Let $L$ be a line bundle on a compact complex manifold $X$. We say that $L$ is pseudo-effective if $L$ can be equipped with a singular Hermitian metric $h$ with $T=\frac{\im}{2\pi}\Theta_{L,h}=\im\pp\varphi\geq0$ as a current, equivalently, the weight $\varphi\in Psh(\Omega)\cap L_{loc}^1(\Omega)$ in any trivalization.
	\end{definition}
	
	Let $\Omega$ be an open subset of $\C^n$. We have the following approxiamation properties for plurisubharmonic functions (see e.g. \cite[Chapter 1 and Chapter 3]{demaillycomplex}).
	
	\begin{lemma}\label{appro}
		Let $u\in Psh(\Omega)\cap L^1_{loc}(\Omega)$. If $(\rho_\epsilon)$ is a family of smoothing kernels, then $u*\rho_\epsilon\in Psh(\Omega_\epsilon)\cap C^\infty(\Omega_\epsilon)$, the family $(u*\rho_\epsilon)$ is decreasing as $\epsilon\rightarrow0$ and $\lim_{\epsilon\rightarrow0}u*\rho_\epsilon=u,$ where $\Omega_\epsilon=\{x\in \Omega:d(x,\p\Omega)>\epsilon\}$.
	\end{lemma}
	
	\begin{lemma}({\cite[Theorem 2.1]{bt82}})\label{bt82}
		Let $\{v_j^1\},\cdots,\{v_j^k\}$ be a decreasing sequences of functions in $Psh(\Omega)\cap L^\infty_{loc}(\Omega)$ and assume that for all $z\in \Omega$,
		$$\lim\limits_{j\rightarrow\infty}v_j^i=v^i\in Psh(\Omega)\cap L^\infty_{loc} (\Omega),\ 1\leq i\leq k.$$
		Then
		$$\lim\limits_{j\rightarrow\infty}v_j^1dd^cv_j^2\wedge\cdots\wedge dd^cv_j^k=v^1dd^cv^2\wedge\cdots dd^cv^k,$$
		where the limits is in the weak topology on the currents of bigdegree $(k,k)$ on of order zero.
	\end{lemma}
	
	\subsection{Some criterions for rational connectedness} Based on the uniruledness criterion estabished in (\cite{BDPP}), Campana, Demailly and Peternell devoloped the following criterion for rational connectedness.
	\begin{lemma}[{\cite[Theorem 1.1]{campana2015}}]\label{cdp}
		Let $M$ be a projective manifold. The following properties are equivalent.
		\begin{enumerate}
			\item[(a)] $M$ is rationally connected.
			\item[(b)] For every invertible subsheaf $\mathcal{F} \subset \mathcal{O}(\Lambda^{p,0} M)$, $1 \leq p \leq n$, $\mathcal{F}$ is not pseudo-effective.
			\item[(c)] For any ample line bundle $A$ on $X$, there exists a constant $C_A > 0$ such that
			\[
			H^0(M, (T^{1,0}M^*)^{\otimes m} \otimes A^{\otimes k}) = 0
			\]
			for all $m, k \in \mathbb{N}^*$ with $m \geq C_A k$.
		\end{enumerate}
	\end{lemma}
	
	\subsection{Some properties of curvature tensor} Let us recall some algebraic properties about curvature tensor.
	\begin{lemma}(Royden's observation, {\cite[P552]{royden1980ahlfors}})\label{royden}
		Let $\xi_1,\cdots,\xi_k$ be $k$ orthonormal tangent vectors. If $S(\xi,\bar{\eta},\zeta,\bar{\w})$ is a symmetric bihermitian form on some complex vector space $V$ (i.e., $S(\xi,\bar{\eta},\zeta,\w)=S(\zeta,\bar{\eta},\xi,\bar{\w})$ and $S(\xi,\bar{\w},\zeta,\bar{\eta})$), such that for all $\xi\in V$, we have
		$$S(\xi,\bar{\xi},\xi,\bar{\xi})\leq K|\xi|^4,$$
		for some constant $K$, then
		$$\sum\limits_{\alpha,\beta}S(\xi_\alpha,\bar{\xi}_\alpha,\xi_\beta,\bar{\xi}_\beta)\leq \frac{1}{2}K[\sum\limits|\xi_\alpha|^4+(\sum\limits|\xi_\alpha|^2)^2].$$
	\end{lemma}
	
	\begin{lemma}({\cite[Lemma 1.6]{yang2020}})\label{L2}
		Let $(M,\w)$ be a compact K\"{a}hler manifold and $x_0$ be an arbitrary point on $M$. Let $e_1\in T_{x_0}^{1,0}M$ be a unit vector which minimizes the holomorphic sectional curvature $H$ of $\w$ at $x_0$, then
		\begin{equation}
			2R(e_1,\overline{e_1},W,\overline{W})\geq(1+\vert\langle W,\overline{e_1}\rangle_g\vert^2)R(e_1,\overline{e_1},e_1,\overline{e_1})
		\end{equation}
		for any unit vector $W\in T_{x_0}^{1,0}M$.
	\end{lemma}
	
	\begin{lemma}({\cite[Lemma 2.1]{heier2018}})\label{L3}
		Let $(M,\w)$ be a compact K\"{a}hler manifold and $x_0$ be an arbitrary point on $M$. Assume that the holomorphic sectional curvature $H$ of $\w$ is nonnegative at $x_0$. If a nonzero tangent vector $V\in T^{1,0}_{x_0}M$ such that $R(V,\overline{V},W,\overline{W})=0$ for any $W\in T^{1,0}_{x_0}(M)$, then
		\begin{equation}
			R(V,\overline{X},Y,\overline{Z})=0
		\end{equation}
		for any $X,Y,Z\in T_{x_0}^{1,0}M$, i.e. $V$ is a nonzero truly flat tangent vector at $x_0$.
	\end{lemma}
	
	\subsection{The structure theorem of morphism} The following theorem reveals a
	detailed structure of morphisms whose domain has semi-positive holomorphic sectional
	curvature.
	\begin{lemma}\cite[Theorem 1.6]{matsumura2022}\label{morp}
		Let \((X, g)\) be a compact K\"ahler manifold with semi-positive holomorphic sectional curvature, and let \( \varphi: X \to Y \) be a surjective morphism to a compact K\"ahler manifold \( Y \) with a pseudo-effective canonical bundle. Then
		\begin{enumerate}
			\item The following statements hold:
			\begin{enumerate}[\textbullet]
				\item The morphism \( \varphi: X \to Y \) is smooth and locally trivial.
				\item The exact sequence of vector bundles
				$
				0 \to T_{X/Y} \to TX \xrightarrow{d\varphi^*} \varphi^* TY \to 0
				$
				admits a holomorphic orthogonal splitting, that is, the orthogonal complement \( T_{X/Y}^\perp \) is a holomorphic vector bundle and there exists an isomorphism \( j: \varphi^* TY \to T_{X/Y}^\perp \) such that it gives the holomorphic orthogonal decomposition
				\[
				TX = T_{X/Y} \oplus j(\varphi^* TY) \cong T_{X/Y} \oplus \varphi^* TY.
				\]
				\item The image \( Y \) admits a K\"ahler metric \( g_Y \) such that \( g_Q = \varphi^* g_Y \) and the holomorphic sectional curvature of \( g_Y \) is identically zero. Here \( g_Q \) is the Hermitian metric on \( \varphi^* TY \) induced by the above exact sequence and the metric \( g \).
			\end{enumerate}
			
			\item Let \( F \) be a fiber of \( \varphi\) and \( F_{\text{univ}} \) be its universal cover. We obtain the isomorphism
			\[
			X_{\text{univ}} \cong \mathbb{C}^m \times F_{\text{univ}},
			\]
			where \( m \) is the dimension of \( Y \). There exists a K\"ahler metric \( g_{F_{\text{univ}}} \) on \( F_{\text{univ}} \) such that the holomorphic sectional curvature of \( g_{F_{\text{univ}}} \) is semi-positive and the above isomorphism \( X_{\text{univ}} \cong \mathbb{C}^m \times F_{\text{univ}} \) is isometric with respect to the K\"ahler metrics induced by \( g \), \( g_Y \), and \( g_{F_{\text{univ}}} \).
		\end{enumerate}
	\end{lemma}

	\section{Proof of Theorem \ref{t2}, Theorem \ref{quasi} and Theorem \ref{sc}}\label{s3}
	In this section, we provide the proofs of Theorem \ref{t2}, Theorem \ref{quasi} and Theorem \ref{sc}. First, we present a proof of Theorem \ref{t2}. The proof of Theorem \ref{t2} is divided into two parts: establishing an integral inequality and performing a pointwise computation. The method is based on integration by parts, Bochner-type formula and algebraic properties of holomorphic sectional curvature.
	\subsection*{Setting.}\label{set}
	Let $(M,\w_g)$ be a compact K\"ahler manifold. Assume that $\mathcal{F}$ is an invertible sheaf of $\mathcal{O}(\Lambda^{p,0}M)$ which is pseudo-effective as a line bundle. Then there exist a finite covering $\{U_\alpha\}$ consists of coordinate neighborhoods and local holomorphic $(p,0)$-forms $\eta_\alpha$ generating $\mathcal{F}|_{U_\alpha}$, $\mathcal{F}$ is associated with the \u{C}ech cocycle $b_{\alpha\beta}$ in $\mathcal{O}_M^*(U_\alpha)$ such that $\eta_\beta = b_{\alpha\beta} \eta_\alpha$ and there exists a singular Hermitian metric $h=e^{-\varphi}$ of $\mathcal{F}$ defined by a collection of plurisubharmonic functions $\varphi_\alpha \in \mathrm{PSH}(U_\alpha)\cap L_{loc}^1(U_\alpha)$ such that $e^{-\varphi_\beta} = |b_{\alpha\beta}|^2 e^{-\varphi_\alpha}$. We can define a globally bounded measurable function
	$
	\psi = e^{\varphi_\alpha} |\eta_\alpha|^2_g
	$
	as in (\cite{campana2015}), which can be viewed as the hermitian metric ratio $\frac{|\eta_\alpha|^2_g}{|\eta_\alpha|^2_h}$. Since $\im\pp$ of the plurisubharmonic function $\varphi_\alpha$ is a closed positive current, We have $$\im\pp(e^{\varphi_\alpha})=e^{\varphi_\alpha}(\im\p\varphi_\alpha\wedge\bp\varphi_\alpha+\im\pp\varphi_\alpha),$$ which is also a closed positive current on $U_\alpha$. This shows that $\p\varphi_{\alpha}$ is $L^2$ with respect to the wight $e^{\varphi_\alpha}$ and $e^{\varphi_\alpha}\pp\varphi_\alpha$ has locally finite measure corfficients on $U_\alpha$. Hence
	$$\p\psi=e^{\varphi_\alpha}(\p\varphi_\alpha\vert\eta_\alpha\vert^2_g+\p\vert\eta_\alpha\vert^2_g)$$
	is $L^2$-integrable on $M$. We will derive the following integral inequality for $\psi$, which is the critical point of the proof of Theorem \ref{t2}.
	
	\begin{proposition}\label{boch}
		The setting is as above, then
		\begin{equation}\label{ine}
			-\int_M|\p\psi|^2_g\cdot\w^n\geq\frac{p^2}{2}\int_M\kappa\psi^2\cdot\w^n
		\end{equation}
		where $\kappa$ is a continuous function such that $\kappa(x)$ which is a lower bound of the holomorphic sectional curvature of $\w$ at $x$.
	\end{proposition}
	
	\subsection{Proof of the integral inequality \eqref{ine}}
	This subsection is devoted to the proof of Proposition \ref{boch}. Firstly, we will show the following key identity for forms.
	\begin{lemma}\label{L1}
		Let $(M, g)$ be a Hermitian manifold of complex dimension $n$, $\alpha$ be a real (1,1)-form  and $\eta$ be a (p,0)-form on $M$. We can define a real semi-positive $(1,1)$-form $\beta$ associated to $\eta$ by
		\begin{equation}
			\beta=\duallefs^{p-1}\left((\sqrt{-1})^{p^2}\frac{\eta\wedge\bar{\eta}}{p!}\right),
		\end{equation}
		where $\duallefs$ is the dual Lefschetz operator with respect to $\w$. Then
		\begin{equation}\label{tr1}
			\tr_\w\beta=\vert \eta\vert^2_g
		\end{equation}
		and
		\begin{equation}\label{wedge}
			\begin{split}
				(\im)^{p^2}\alpha\wedge \eta\wedge\bar{\eta}\wedge \frac{\w^{n-p-1}}{(n-p-1)!}=\left[\tr_\w\alpha\cdot\vert \eta\vert_g^2-p\langle \alpha,\beta\rangle_g\right]\frac{\w^n}{n!},
			\end{split}
		\end{equation}
		where $\w$ is the fundamental (1,1)-form associated to the Hermitian metric $g$.
	\end{lemma}
	
	\begin{proof}
		In local holomorphic coordinates $\{z^{1}, \cdots , z^{n}\}$, we have $\alpha=\im\alpha_{j\bar{k}}dz^{j}\wedge d\bar{z}^k$, $\w=\im g_{j\bar{k}}dz^j\wedge d\bar{z}^k$ and $\eta=\eta_{I_p}dz^{I_p}$, where $I_p=(i_1,\cdots,i_p)\in \N_+^p$ and $ dz^{I_p}=dz^{i_1}\wedge\cdots\wedge dz^{i_p}$.
		For any $x_0\in M$, we can choose local holomorphic coordinates centered at $x_{0}$ such that
		$$\w=\im\sum\limits_idz^i\wedge d\bar{z}^i$$
		and
		$$\alpha=\im\sum\limits_i\alpha_{i\bar{i}}dz^i\wedge d\bar{z}^i$$
		at $x_0$.	So, we have
		\begin{align*}
			\w^{n-p-1}=(\im)^{n-p-1}(n-p-1)!\sum\limits_{i_1<i_2<\cdots<i_{n-p-1}}dz^{i_1}\wedge d\bar{z}^{i_1}\wedge\cdots\wedge dz^{i_{n-p-1}}\wedge d\bar{z}^{i_{n-p-1}}
		\end{align*}
		and
		\begin{align*}
			(\im)^{p^2}\alpha\wedge \frac{\eta\wedge\bar{\eta}}{p!}
			=&(\im)^{p+1}\sum\limits_{J_p}\sum\limits_{i\notin J_p}\alpha_{i\bar{i}}\eta_{J_p}\overline{\eta_{J_p}}dz^i\wedge d\bar{z}^{i}\wedge dz^{j_1}\wedge d\bar{z}^{j_1}\wedge\cdots\wedge dz^{j_p}\wedge d\bar{z}^{j_p}\\
			&+(\text{terms involving $\eta_{J_p}\overline{\eta_{K_p}}$ with $J_p\neq \pi(K_p)$ for some permutation $\pi$}).
		\end{align*}
		It follows that
		\begin{equation}
			\begin{split}
				&(\im)^{p^2}\frac{n!}{(n-p-1)!}\alpha\wedge \frac{\eta\wedge\bar{\eta}}{p!}\wedge \w^{n-p-1}\\=&(\im)^nn!\sum\limits_{J_p}\sum\limits_{i\notin J_p}\alpha_{i\bar{i}}\eta_{J_p}\overline{\eta_{J_p}}dz^1\wedge d\bar{z}^1\wedge\cdots\wedge dz^n\wedge d\bar{z}^n\\
				=&\left(\sum\limits_{J_p}\sum\limits_{i\notin J_p}\alpha_{i\bar{i}}\eta_{J_p}\overline{\eta_{J_p}}\right)\w^n\\
				=&\left(\sum\limits_{i,J_p}\alpha_{i\bar{i}}\eta_{J_p}\overline{\eta_{J_p}}-\sum\limits_{i\in J_p}\alpha_{i\bar{i}}\eta_{J_p}\overline{\eta_{J_p}} \right)\w^n\\
				=&\left(\sum\limits_{i,J_p}\alpha_{i\bar{i}}\eta_{J_p}\overline{\eta_{J_p}}-p\sum\limits_{i,I_{p-1}}\alpha_{i\bar{i}}\eta_{iI_{p-1}}\overline{\eta_{iI_{p-1}}} \right)\w^n\\
				=&\left[\frac{\tr_\w\alpha\cdot\vert \eta\vert_g^2}{p!}-\langle \alpha\wedge\w^{p-1},(\im)^{p^2}\frac{\eta\wedge\bar{\eta}}{p!(p-1)!}\rangle_g\right]\w^n\\
				=&\left[\frac{\tr_\w\alpha\cdot\vert \eta\vert^2_g}{p!}-\frac{1}{(p-1)!}\langle \alpha,\duallefs^{p-1}\left((\sqrt{-1})^{p^2}\frac{\eta\wedge\bar{\eta}}{p!}\right)\rangle_g\right]\w^n\\
				=&\left[\frac{\tr_\w\alpha\cdot\vert \eta\vert^2_g}{p!}-\frac{1}{(p-1)!}\langle \alpha,\beta\rangle_g\right]\w^n.
			\end{split}
		\end{equation}
		This verifies \eqref{wedge}. Locally, $\beta=\sqrt{-1}\beta_{i\bar{j}}dz^i\wedge d\bar{z}^j$ and
		\begin{equation}
			\beta_{i\bar{j}}=p!\sum\limits_{I_{p-1}}\eta_{iI_{p-1}}\overline{\eta_{jI_{p-1}}}.
		\end{equation}
		We can reselect the local holomorphic coordiates such that $(g_{i\bar{j}})$ is still the identity and $(\beta_{i\bar{j}})$ is diagonal at $x_0$. Then
		\begin{equation}\label{beta}
			\beta_{i\bar{i}}=p!\sum\limits_{I_{p-1}}\eta_{iI_{p-1}}\overline{\eta_{iI_{p-1}}}\geq0.
		\end{equation}
		Hence $\beta$ is semi-positive and
		\begin{equation}
			\tr_\w\beta=p!\sum\limits_{i,I_{p-1}}\eta_{iI_{p-1}}\overline{\eta_{iI_{p-1}}}=\vert \eta\vert^2_g.
		\end{equation}
		The proof is completed.
	\end{proof}

	\begin{proof}[\bf Proof of Proposition \ref{boch}] Now let us proceed to show Proposition \ref{boch}. It follows from Lemma \ref{L1} that
		\begin{equation}\label{01}
			\begin{split}
				n&\psi\im\pp\psi\wedge\w^{n-1}=e^{\varphi_\alpha}\tr_\w(\im\pp\psi)|\eta_\alpha|^2_g\cdot\w^n\\
				=&p\langle \im\pp \psi,e^{\varphi_\alpha}\beta_\alpha\rangle_g\cdot\w^n+(\im)^{p^2}\frac{n!}{(n-p-1)!}\im \pp \psi\wedge e^{\varphi_\alpha} \eta_\alpha\wedge\overline{\eta_\alpha}\wedge\w^{n-p-1}
			\end{split}
		\end{equation}
		where $\beta_\alpha=\Lambda^{p-1}((\im)^{p^2}\frac{\eta_\alpha\wedge\overline{\eta_\alpha}}{p!})$. $e^{\varphi_\alpha}\beta_{\alpha}$ and $e^{\varphi_\alpha}\eta_{\alpha}\wedge\overline{\eta_\alpha}$ are globally defined differential forms with bounded measurable cofficients. For clarity, the arguement will be divided into three steps.
		
		\subsection*{Step 1. The integration by parts} Although $\varphi_\alpha$ is in general singular, we can still perform integration by parts as in \cite{de02}. For the reader's convenience, we provide a detailed explaination.
		
		Consider a partitions of unity $\{\rho_\alpha\}$ subordinated to the finite cover $\{U_\alpha\}$. Since $\supp\rho_\alpha\subset\subset U_\alpha$, for any $\epsilon>0$ sufficiently small, there exists a sequence of smooth plurisubharmonic functions $\varphi_{\alpha,\epsilon}$ decreasing to $\varphi_{\alpha}$ on some open subset $V_\alpha$ such that $\supp\rho_\alpha\subset\subset V_\alpha\subset\subset U_\alpha$. Applying the results of Bedford-Taylor (Lemma \ref{bt82}) to the uniformly bounded function $e^{\varphi_{\alpha,\epsilon}}$ on $V_\alpha$, we obtain local weak covergence $$e^{\varphi_{\alpha,\epsilon}}\pp\varphi_{\alpha,\epsilon}\rightarrow e^{\varphi_\alpha}\pp\varphi_\alpha,\ e^{\varphi_{\alpha,\epsilon}}\rightarrow e^{\varphi_\alpha}\ \text{and}\  e^{\varphi_{\alpha,\epsilon}}\p\varphi_{\alpha,\epsilon}\wedge\bp\varphi_{\alpha,\epsilon}\rightarrow e^{\varphi_{\alpha}}\p\varphi_{\alpha}\wedge\bp\varphi_{\alpha}$$ on $V_\alpha$. Thus
		\begin{equation}\label{idp1}
			\begin{split}
				&-\int_M\left|\p\psi\right|^2_g\cdot\w^n=-n\int_M\p\psi\wedge\bp\psi\wedge\sum\limits_\alpha\rho_\alpha\w^{n-1}\\
				=&-n\sum\limits_\alpha\int_{V_\alpha}\im\p(e^{\varphi_\alpha}|\eta_\alpha|_g^2)\wedge\bp(e^{\varphi_\alpha}|\eta_\alpha|_g^2)\wedge\rho_\alpha\w^{n-1}\\
				=&-n\sum\limits_\alpha\lim\limits_{\epsilon\rightarrow0}\int_{V_\alpha}\im\p(e^{\varphi_\alpha,\epsilon}|\eta_\alpha|_g^2)\wedge\bp(e^{\varphi_{\alpha,\epsilon}}|\eta_\alpha|_g^2)\wedge\rho_\alpha\w^{n-1}\\
				=&n\sum\limits_\alpha\lim\limits_{\epsilon\rightarrow0}\int_{V_\alpha}e^{\varphi_\alpha,\epsilon}|\eta_\alpha|_g^2\left[\im\pp(e^{\varphi_{\alpha,\epsilon}}|\eta_\alpha|_g^2)\wedge\rho_\alpha\w^{n-1}+\p\rho_\alpha\wedge\im\bp(e^{\varphi_{\alpha,\epsilon}}|\eta_\alpha|_g^2)\wedge\w^{n-1}\right]\\
				=&n\sum\limits_\alpha\int_{V_\alpha}e^{\varphi_\alpha}|\eta_\alpha|_g^2\left[\im\pp(e^{\varphi_{\alpha}}|\eta_\alpha|_g^2)\wedge\rho_\alpha\w^{n-1}+\p\rho_\alpha\wedge\im\bp(e^{\varphi_{\alpha}}|\eta_\alpha|_g^2)\wedge\w^{n-1}\right]\\
				=&n\int_M\psi\im\pp\psi\wedge\w^{n-1}.
			\end{split}
		\end{equation}
		The last equality follows from the fact $\sum\limits_\alpha\rho_\alpha=1$. Combining \eqref{idp1} with \eqref{01}, we obtain
		\begin{equation}\label{ibp}
			\begin{split}
				0\geq&-\int_M|\p\psi|_g^2\cdot\w^n\\
				=&p\int_M\langle \im\pp \psi,e^{\varphi_\alpha}\beta_\alpha\rangle_g\cdot\w^n+(\im)^{p^2}\frac{n!}{(n-p-1)!}\int_M\im \pp \psi\wedge e^{\varphi_\alpha} \eta_\alpha\wedge\overline{\eta_\alpha}\wedge\w^{n-p-1}.
			\end{split}
		\end{equation}
		Next, we analyze each term on the right-hand side of integral inequality \eqref{ibp} seperately. First, we consider the sign of first term in \eqref{ibp}.
		
		\subsection*{Step 2. The Bochner formula.}  For any $x_0\in U_\alpha$, choose local holomorphic coordinates centered at $x_0$ such that
		$$\w=\im\sum\limits_{i}dz^i\wedge d\bar{z}^i$$
		and
		$$\beta_\alpha=\sqrt{-1}\sum\limits_i\beta_{i\bar{i}}dz^i\wedge d\bar{z}^i$$
		at $x_0$. For simplicity, we omit the index $\alpha$ in the following calculation. We write $\eta=\sum\limits \eta_{I_p}dz^{I_p}$. Recall \eqref{tr1} and \eqref{beta}, which give us
		\begin{equation}
			\tr_\w\beta=\vert \eta\vert_g^2
		\end{equation}
		and
		\begin{equation}
			\beta_{i\bar{i}}=p!\sum\limits_{I_{p-1}}\eta_{iI_{p-1}}\overline{\eta_{iI_{p-1}}}\geq0.
		\end{equation}
		By direct calculation, we derive the following Bochner type formula in the sense of distribution,
		\begin{equation}\label{Bochner}
			\begin{split}
				\p_u\p_{\bar{u}}(e^\varphi\vert \eta\vert^2_g)=&e^\varphi(\vert\eta\vert^2\p_u\p_{\bar{u}}\varphi+\vert D'_u \eta+\eta\p_u\varphi\vert_{\Lambda^{p,0}M}^2+p!\sum\limits_{I_p}\sum\limits_{k=1}^p\sum\limits_{l=1}^nR_{u\bar{u}i_k\bar{l}}\eta_{I_p}\overline{\eta_{i_1\cdots(l)_k\cdots i_p}})\\
				=&e^\varphi(\vert\eta\vert^2\p_u\p_{\bar{u}}\varphi+\vert D'_u \eta+\eta\p_u\varphi\vert_{\Lambda^{p,0}M}^2+p(p!)\sum\limits_{l,k=1}^nR_{u\bar{u}l\bar{k}}\sum\limits_{J_{p-1}}\eta_{lJ_{p-1}}\overline{\eta_{kJ_{p-1}}})\\
				=&e^\varphi(\vert\eta\vert^2\p_u\p_{\bar{u}}\varphi+\vert D'_u \eta+\eta\p_u\varphi\vert_{\Lambda^{p,0}M}^2+p\sum\limits_{l,k}R_{u\bar{u}l\bar{k}}\beta_{l\bar
					{k}})
			\end{split}
		\end{equation}
		for all $u\in T^{1,0}_{x_0}M$, where $R_{i\bar{j}k\bar{l}}=R(\p_i,\overline{\p_j},\p_k,\overline{\p_l})$ and $D'$ is the $(1,0)$-part of the connection $D$ of $\Lambda^{p,0}M$ inducced by Chern connection of $(T^{1,0}M,\w)$. The second equality follows from interchanging the order of summation. Hence
		\begin{equation}\label{02}
			\begin{split}
				&p\left\langle \im\pp\psi,e^\varphi\beta\right\rangle_g=e^{2\varphi}(p\vert\eta\vert^2\p_i\p_{\bar{i}}\varphi\beta_{i\bar{i}}+p\sum\limits_i \vert D'_i \eta+\eta\p_i\varphi\vert^2\beta_{i\bar{i}}+p^2\sum\limits_{i,k}R_{i\bar{i}k\bar{k}}\beta_{i\bar{i}}\beta_{k\bar{k}})\\
				=&e^{2\varphi}\left(p\psi\langle\im\pp\varphi,\beta\rangle+p\left<
				\im\langle D' \eta+\eta\p\varphi,\overline{D' \eta+\eta\p\varphi}\rangle_g,\beta\right>_g+p^2\langle \Rm,-\beta\otimes\beta\rangle_g\right).
			\end{split}
		\end{equation}
		It follows from Royden's observation(Lemma \ref{royden}) that
		\begin{equation}
			\begin{split}
				e^{2\varphi}\langle \Rm,-\beta\otimes\beta\rangle_g=e^{2\varphi}\sum\limits_{i,k}R_{i\bar{i}k\bar{k}}\beta_{i\bar{i}}\beta_{k\bar{k}}\geq& e^{2\varphi}\frac{\ka(x_0)}{2}\left[(\sum\limits_{i}\beta_{i\bar{i}})^2+\sum\limits_{i}\beta_{i\bar{i}}^2 \right]\\
				=&\frac{\ka(x_0)}{2}e^{2\varphi}\left(\vert\eta\vert_g^4+\vert\beta\vert_g^2\right)\\
				\geq& \frac{\kappa(x_0)}{2}\psi^2
			\end{split}
		\end{equation}
		at $x_0$. The fact that $\im\pp\varphi\geq0$ and $\beta\geq0$ implies that
		\begin{equation}
			e^{2\varphi}(p\psi\langle\pp\varphi,\beta\rangle+p\left<
			\im\langle D' \eta+\eta\p\varphi,\overline{D' \eta+\eta\p\varphi}\rangle_g,\beta\right>_g\geq0.
		\end{equation}
		Hence, we have
		$$p\int_M\left\langle \im\pp\psi,e^\varphi\beta\right\rangle_g\cdot\w^n=p\int_{M}\left\langle \im\pp\psi,e^\varphi\beta\right\rangle_g\cdot\w^n\geq\frac{p^2}{2}\int_{M}\kappa\psi^2\cdot\w^n.$$
		
		\subsection*{Step 3. The sign of the second term of \eqref{ibp}:} As
		\begin{align*}
			&\pp(e^{\varphi_\alpha}\eta_\alpha\wedge \overline{\eta_\alpha})=\p(e^{\varphi_\alpha}\bp\varphi_\alpha\wedge\eta_\alpha\wedge\overline{\eta_\alpha}+e^{\varphi_\alpha}(-1)^p\eta_\alpha\wedge\overline{\p\eta_\alpha})\\
			=&(\im)^{2p}e^{\varphi_\alpha}(\p\eta_\alpha+\p\varphi_\alpha\wedge\eta_\alpha)\wedge\overline{(\p\eta_\alpha+\p\varphi_\alpha\wedge\eta_\alpha)}+e^{\varphi_\alpha}\pp\varphi_\alpha\wedge\eta_\alpha\wedge\overline{\eta_\alpha},
		\end{align*}
		by applying the integration by parts which is valid as previously explained, we get
		\begin{align*}
			&(\im)^{p^2}\int_M\im \pp \psi\wedge e^{\varphi_\alpha} \eta_\alpha\wedge\bar{\eta}_\alpha\wedge\w^{n-p-1}\\
			=&(\im)^{p^2}\int_M\psi\im\pp(e^{\varphi_\alpha}\eta_\alpha\wedge\overline{\eta}_\alpha\wedge\w^{n-p-1})\\
			=&(\im)^{p^2}\int_M\psi\im\pp(e^{\varphi_\alpha}\eta_\alpha\wedge\overline{\eta}_\alpha)\wedge\w^{n-p-1}\\
			=&\int_M\psi (\im)^{(p+1)^2}e^{\varphi_\alpha}\gamma_\alpha\wedge\overline{\gamma_\alpha}\wedge\w^{n-p-1}+\int_M\psi\im\pp\varphi_\alpha\wedge e^{\varphi_\alpha}(\im)^{p^2}\eta_\alpha\wedge\bar{\eta}_\alpha\wedge\w^{n-p-1}\\
			\geq&0.
		\end{align*}
		where $\gamma_\alpha=\p\eta_\alpha+\p\varphi_\alpha\wedge\eta_\alpha$. The last inequality follows from the fact $$\im\pp\varphi_\alpha\geq0,\ (\im)^{(p+1)^2}\gamma_\alpha\wedge\overline{\gamma_\alpha}\geq0 \text{  and  } (\im)^{p^2}\eta_\alpha\wedge\bar{\eta}_\alpha\geq0.$$
		In summary, we get the desired integral inequality. The proof is completed.
	\end{proof}
	
	\subsection{The proof of Theorem \ref{t2}}\label{s3.3}
	In this subsection, we proceed to show Theorem \ref{t2}. The method is based on performing a pointwise computation.
	\begin{proof}[\bf Proof of Theorem \ref{t2}]
		Let $\mF$ be an invertible subsheaf of $\mO({\Lambda^{p,0}M})$ that is pseudo-effective. Since the holomorphic sectional curvature of $\w$ is nonnegative on $M$, we can choose $\kappa\equiv0$ in the setting \ref{set}. By Proposition \ref{boch}, we have
		$$-\int_M|\p\psi|_g^2\cdot\w^n\geq0.$$	
		Hence $d\psi=0$ almost everywhere on $M$, which implies that $\psi= C$ for some constant $C\geq0$ almost everywhere.
		We claim that $C>0$. Suppose not, we have $\psi=0$ almost everywhere on $U_\alpha$. Since $\varphi_\alpha\in Psh(U_\alpha)\cap L_{loc}^1(U_\alpha)$, $\{\varphi_\alpha=-\infty\}$ has zero Lebesgue measure. This fact	implies that $\eta_\alpha\equiv0$ on $U_\alpha$ due to the smoothness of $\eta_\alpha$, which contradicts the fact that $\eta_\alpha$ generates $\mF|_{U_\alpha}$. Hence, $\varphi_\alpha<+\infty$ and $e^{\varphi_\alpha}|\eta_\alpha|^2_g=C>0$ implies that $\vert\eta_\alpha\vert^2_g>0$ and $\varphi_\alpha=\log C-\log|\eta_\alpha|^2_g$ is a smooth function on $U_\alpha$. The following calculation will be done in the smooth sense.
		
		For any $\alpha$, take an arbitrary point $x_0$ of $U_\alpha$. For simplicity, we omit the index $\alpha$. Recall that we can define a real semi-positive $(1,1)$-form $\beta$ associated to $\eta$ by
		\begin{equation}
			\beta=\duallefs^{p-1}\left((\sqrt{-1})^{p^2}\frac{\eta\wedge \overline{\eta}}{p!}\right)
		\end{equation}
		such that
		\begin{equation}
			\tr_\w\beta=\vert \eta\vert^2_g>0.
		\end{equation}
		We choose local holomorphic coordinates centered at $x_0$ such that $(g_{i\bar{j}})$ is the identity and $(\beta_{i\bar{j}})$ is diagonal at $x_0$. Since $\beta_{i\bar{i}}$ are nonnegative for all $1\leq i\leq n$ and $\sum\limits_i\beta_{i\bar{i}}>0$, without loss of generality, we may assume that
		\begin{equation}
			\beta_{1\bar{1}}\geq\beta_{2\bar{2}}\geq\cdots\geq\beta_{m\bar{m}}>0=\beta_{m+1\overline{m+1}}=\cdots=\beta_{n\bar{n}}
		\end{equation}
		for some positive integer $m\leq n$. Denote $$\lambda_j=\sqrt{\beta_{j\bar{j}}}>0,\ 1\leq j\leq m.$$
		From the Bochner type formula \eqref{Bochner}, we have
		\begin{align}\label{eq1}
			0=e^{-\varphi}\p_v\p_{\bar{v}}(e^\varphi|\eta|^2_g)=|\eta|^2_g\p_v\p_{\bar{v}}\varphi+\vert D'_v\eta+\eta\p_v\varphi|^2+p\sum\limits_{j=1}^mR_{v\bar{v}j\bar{j}}\beta_{j\bar{j}}\geq p\sum\limits_{j=1}^mR_{v\bar{v}j\bar{j}}\beta_{j\bar{j}}\lambda_j^2
		\end{align}
		for all $v\in T_{x_0}^{1,0}M$. Let $t=(t_1,\cdots,t_m)$ be the complex coordinates of $\C^m$ and $d\theta$ be the standard volume measure of the unit sphere $S_{2m-1}:=\{t\in \C^m:\sum\limits_{i=1}^m\vert t_i\vert^2=1 \}$. By the Berger's averaging trick (\cite{berger}), we get
		\begin{equation}\label{eq2}
			\sum\limits_{i,j,k,l=1}^m\fint\limits_{ S_{2m-1}} R_{i\bar{j}k\bar{l}}\lambda_i t_i\overline{\lambda_j} \overline{t_j}\lambda_k t_k\overline{\lambda_l} \overline{t_l} d\theta(t)=\frac{2}{m(m+1)}\sum\limits_{i,j=1}^mR_{i\bar{i}j\bar{j}}\lambda_i^2\lambda_j^2\leq 0.
		\end{equation}
		For any $t\in S_{2m-1}$,
		\begin{equation}\label{eq3}
			\sum\limits_{i,j,k,l=1}^mR_{i\bar{j}k\bar{l}}\lambda_i t_i\overline{\lambda_j} \overline{t_j}\lambda_k t_k\overline{\lambda_l} \overline{t_l}=R(V_t,\overline{V_t},V_t,\overline{V_t})\geq0,
		\end{equation}
		where $V_t:=\sum\limits_{i=1}^m\lambda_it_i\frac{\p}{\p z^i}$. Combining \eqref{eq2} with \eqref{eq3}, $R(V_t,\overline{V_t},V_t,\overline{V_t})=0,\ \forall t\in S_{2m-1}$. Taking $t_0\in S_{2m-1}$ such that only the $l$-th component is $1$ and all other components are 0, then $V_{t_0}=\lambda_l\frac{\p}{\p z^l}$ and
		\begin{equation}\label{eq4}
			R_{l\bar{l}l\bar{l}}=\frac{1}{\lambda_l^4}R(V_{t_0},\overline{V_{t_0}},V_{t_0},\overline{V_{t_0}})=0,\ \forall 1\leq l\leq m.
		\end{equation}
		Since the holomorphic sectional curvature is nonnegative, $\frac{\p}{\p z^l}$ minimizes the holomorphic sectional curvature of $\w$ at $x_0$. It follows from Lemma \ref{L2} that
		$$R_{l\bar{l}v\bar{v}}\geq0,\ \forall 1\leq l\leq m,v\in T^{1,0}_{x_0}M.$$
		Since $\p_v\p_{\bar{v}}\varphi\geq0$, \eqref{eq1} implies that
		$R_{l\bar{l}v\bar{v}}=0,\ \forall 1\leq l\leq m$
		and $$\p_v\p_{\bar{v}}\varphi=0,\ D'_v\eta+\eta\p_v\varphi=0$$
		for all $v\in T^{1,0}_{x_0}M.$
		Hence $\{\frac{\p}{\p z^i},\ 1\leq i\leq m\}$ are nonzero truly flat tangent vectors at $x_0$ by Lemma \ref{L3}.
		
		Since $x_0\in M$ is arbitrary, we have $\pp\varphi=0$ on $M$ and $D'\eta+\eta\p\varphi=0$ on $M$. The fact that $\pp\varphi_\alpha=0$ implies that $(\mF,e^{-\varphi})$ is a flat line bundle. In each coordinate neighborhood $U_\alpha$, $\varphi_\alpha$ is pluriharmonic, so it can be written as $\varphi_\alpha=f_\alpha+\overline{f_\alpha}$ for some holomorphic function $f_\alpha\in\mO(U_\alpha)$. Therefore,
		\begin{equation}\label{para}
			D(e^{f_\alpha}\eta_\alpha)=e^{f_\alpha}(D'\eta_\alpha+\eta_\alpha\p\varphi_\alpha)=0
		\end{equation}
		where $e^{f_\alpha}\eta_\alpha$ is a local holomorphic section of $\mF|_{U_\alpha}$. This shows that $\mF$ is invariant by parallel transport induced by connection $D$ of $\Lambda^{p,0}M$, which verifies the statement (1).
		
		We obtain a global real semi-positive $(1,1)$-form given by $$\beta:=e^{\varphi_\alpha}\beta_\alpha=\Lambda^{p-1}\left((\im)^{p^2}\frac{e^{f_\alpha\eta_\alpha}\wedge\overline{e^{f_\alpha}\eta_\alpha}}{p!}\right),$$
		which is determined by $\mF$ and is parallel with respect to $D$. We then define a parallel and self-adjoint linear transformation $S: T^{1,0}M\rightarrow T^{1,0}M$ by
		\begin{equation}\label{trans}
			\langle S(X),Y \rangle_g=\beta(X,\bar{Y}),\ \forall X,Y\in T^{1,0}M.
		\end{equation}
		From the above computation at a fixed point $x_0$, $V_{x_0}=\spaned\{\frac{\p}{\p z^i}\}_{i=1}^m$ and $W_{x_0}=\spaned\{\frac{\p}{\p z^i}\}$ consist of eigenvectors corresponding to nonzero and zero eigenvalues of $S_{x_0}$, respectively. We also have that every $v\in T_{x_0}^{1,0}M$ is truly flat and $\dim V_{x_0}=m>0$. Since $S$ is parallel, $x\mapsto V_x$ and $x\mapsto W_x$ define parallel distrubutions. The proof is complete.
	\end{proof}
	
	\subsection{Proof of Theorem \ref{quasi} and Theorem \ref{sc}}
	In this subsection, we proceed to show Theorem \ref{quasi} and Theorem \ref{sc}. The proofs are based on Theorem \ref{t2}, Campana-Demailly-Peternell's critetion, de Rham's splitting theorem (\cite{dr52}).
	
	\begin{proof}[\bf Proof of Theorem \ref{quasi}]
		Since the holomorphic sectional curvature is nonnegative and there is no nonzero truly flat tangent vector at some point, by Theorem \ref{t2} any invertible subsheaf $\mathcal{F}\subset\mathcal{O}(\Lambda^{p,0}M)$ is not pseudo-effective. This actually implies that $H^{p,0}_{\bp}(M)=0,\ \forall 1\leq p\leq n$. In particular, M is projective by Kodaira's theorem (\cite{Kod}). Combining this with Campana-Demailly-Peternell's criterion (Lemma \ref{cdp}), we conclude that $M$ is rationally connected.
	\end{proof}
	
	\begin{proof}[\bf Proof of Theorem \ref{sc}]
		We will argue by contradiction. Suppose there exists an invertible subsheaf $\mF$ of $\mO(\Lambda^{p,0}M)$ that is pseudoeffective as a line bundle. According to Theorem \ref{t2}, $T^{1,0}M$ admits a decomposition $V+W$ of $T^{1,0}M$ where $V$ and $W$ consist of eigenvectors corresponding to nonzero and zero eigenvalues of $S$, respectively, and $\forall p\in M$, every $v\in V_p$ is truly flat. Since $M$ is simply connected, by de Rham's splitting theorem (\cite[Theorem II]{dr52}), $(M,\w)\simeq (M_1,\w_1)\times (M_2,\w_2)$ isometrically and holomorphically, with $TM_1=V$. This implies that $M_1$ is a simply connected K\"ahler manifold with the flat metric $\w_1$. According to the classification of complex space forms, $M_1$ must be $\C^m$ with $m=\dim V>0$, which contradicts the fact that $M$ is compact. By Campana-Demailly-Peternell's criterion (Lemma \ref{cdp}), the proof is complete.
	\end{proof}
	
	\section{Proof of Theorem \ref{lowdim}}\label{23}
	In this section, we provide a proof of Theorem \ref{lowdim}. The proof relies on $L^2$-vanishing theorem (Proposition \ref{L2}), Atiyah's $L^2$-index Theorem (\cite{atiyah76}) and properties of Albanese maps. First, we present the following $L^2$-vanishing theorem for nonnegative holomorphic sectional curvature.
	\begin{proposition}\label{L2}
		Let $(M,\w)$ be a complete K\"ahler manifold with nonnegative holomorphic sectional curvature, then any $L^2$-holomorphic $(p,0)$-form is parallel.
	\end{proposition}
	
	The idea of proof is essentially the same as that for Theorem \ref{t2}, but with some minor differences, because $|\eta|^4_g$ may be not $L^1$-integrable. For the reader's convenience, we provide a brief proof in the appendix. From Proposition \ref{L2} and Atiyah's $L^2$-index theorem (\cite{atiyah76}), we derive the following proposition.
	\begin{proposition}\label{cha}
		Let $(M,\w)$ be a compact K\"ahler manifolds with nonnegative holomorphic sectional curvature. One of the two following situations occurs:
		\begin{itemize}
			\item[(\romannumeral1)] $\pi_1(M)$ is finite, then $M$ is projective and rationally connected.
			\item[(\romannumeral2)] $\pi_1(M)$ is infinite, then $\chi(M,\mO_X)=0$.
		\end{itemize}
	\end{proposition}
	
	\begin{proof}
		Let $f:\widetilde{M}\rightarrow M$ be the universal cover of $M$, then $(\widetilde{M},f^*\w)$ is a compact simple connected K\"ahler manifold with nonnegative holomorphic sectional curvature. If $\pi_1(M)$ is finite, then $(\widetilde{M},f^*\w)$ is compact simple connected, hence projective and rationally connected (Theorem \ref{sc}). Since a pseudoeffective invertible subsheaf $\mF$ of $\mO(\Lambda^{p,0}M)$ induces a pseudoeffective invertible subsheaf $f^*\mF$ of $\mO(\Lambda^{p,0}\widetilde{M})$, $M$ must be projective and rationally connected by the Compana-Demailly-Peternell's criterion (Lemma \ref{cdp}).
		
		If $\pi_1(M)$ is infinite, then $(\widetilde{M},f^*\w)$ is noncompact and complete. Denote the Hilbert space of $L^2$-harmonic form of degree $(p,q)$ by $H^{p,q}_{L^2}(\widetilde{M})$. According to $L^2$-Hodge theory (see \cite{gromov91}), a $L^2$ $(p,0)$-form is harmonic if and only if it is a holomorphic $(p,0)$-form. In general, $H^{p,q}_{L^2}(\widetilde{M})$ might be infinite-dimensional. Using the isometric action of $\Gamma:=\pi_1(X)$ on them, one can associate to them a nonnegative real number $\dim_\Gamma(H^{p,q}_{L^2}(\widetilde{M}))$ (c.f. \cite{atiyah76}). $\dim_\Gamma(H^{p,q}_{L^2}(\widetilde{M}))=0$ if and only if $H^{p,q}_{L^2}(\widetilde{M})=\{0\}$. By Atiyah's $L^2$-index theorem (c.f. \cite{atiyah76,gromov91}), we know that
		$$\chi(M,\mO_M)=\chi_{L^2}(\widetilde{M},\mO_{\widetilde{M}}):=\sum\limits_{k=0}^{\dim M}(-1)^k\dim_\Gamma(H^{0,q}_{L^2}(\widetilde{M})).$$
		By Proposition \ref{L2}, the norm $|\eta|_{f^*g}$ of any $L^2$-holomorphic $(p,0)$-form $\eta$ is equal to some constant. Because the volume of $(\widetilde{M},f^*\w)$ is infinite, any $L^2$-holomorphic $(p,0)$-form must be zero. Thus $H^{p,0}_{L^2}(M)=\{0\}$ for every $0\leq p\leq n$. By conjugation, $H^{0,p}_{L^2}(M)=\{0\}$ for every $0\leq p\leq n$. Consequently, $\chi(M,\mO_M)=0$.
	\end{proof}
	
	From Proposition \ref{L2}, we have the following fact about the Albanese map.
	\begin{proposition}\label{alb}
		Let $(M,\w)$ be a compact K\"ahler manifold with nonnegative holomorphic sectional curvature. Then the Albanese map $\alpha:=M\rightarrow A(M)$ is a submersion. In particular, it is surjective.
	\end{proposition}
	\begin{proof}
		Let \( q = \dim H^{1,0}_{\bp}(M) \) and let \(\{\eta_1, \cdots, \eta_q\}\) be a basis of holomorphic \((1,0)\)-forms. Recall that the Albanese map \(\alpha\) is defined by
		\[
		x \mapsto \left(\int_a^x \eta_1, \cdots, \int_a^x \eta_q \right) \mod \Lambda,
		\]
		where \(\Lambda \subset \mathbb{C}^q\) (see, e.g., \cite[Chapter \text{VI}, section 9.2]{demaillycomplex}). Then \(d\alpha = (\eta_1, \cdots, \eta_q)\) at every point. If \(\rank(d\alpha) < q = \dim(A(M))\) at some point \(x\), then there would exist a non-zero linear combination
		\[
		u = \lambda_1 \eta_1 + \cdots + \lambda_q \eta_q
		\]
		with \(u(x) = 0\). By Proposition \ref{L2}, \(u\) must be zero, which leads to a contradiction.
	\end{proof}
	
	Now let us proceed to show Theorem \ref{lowdim}.
	\begin{proof}[\bf Proof of Theorem \ref{lowdim}]
		Since $M$ is non-projective, it follows from Kodaira's theorem (\cite{Kod}) that $M$ admits a nonzero holomorphic $(2,0)$-form, and from Proposition \ref{cha} that $\chi(M,\mO_M)=0$. Let $h^{p,0}$ denote $\dim H^{p,0}_{\bp}(M)$.
		
		If $h^{3,0}>0$, there admits a nonzero holomorphic $(3,0)$-form $\eta$. By Proposition \ref{L2}, $\eta$ must be parallel and therefore doesn't vanish anywhere. Since $\dim M=3$, the existence of the nowhere-vanishing holomorphic $(3,0)$-form $\eta$ implies that $K_M\simeq \mO_M$. Consequently, $K_M$ is flat, and we have:
		$$0=-n\int_Mc_1(K_M)\wedge\w^{n-1}=n\int_M\rc(\w)\wedge\w^{n-1}=\int_MS\cdot\w^n,$$
		where $S$ is the scalar curvature of $\w$. According to Berger's averaging trick (\cite{berger}),
		$$S(x)=\fint_{V\in T^{1,0}_xM, |V|_g=1}H(V)d\theta(V)\geq0$$
		at any point $x$ of $M$. Hence, the holomorphic sectional curvature is identically zero, and thus $M$ is flat.
		
		If $h^{3,0}=0$, then $1-h^{1,0}+h^{2,0}=\chi(M,\mO_M)=0$, which implies that $h^{1,0}=1+h^{2,0}\geq2$. Suppose that $h^{1,0}=3$. Then we have $h^{3,0}\geq1$, which is a contradiction. Therefore, $h^{1,0}=2$. From Matsumura's result (Lemma \ref{morp}) and Proposition \ref{alb}, the Albanese map $\alpha: M\rightarrow \T^2$ is locally trivial and smooth. The fibre $F$ of $\alpha$ admits a K\"ahler metric $\w_F$ with nonnegative holomorphic sectional curvature, and the universal cover $(\widetilde{M},\widetilde{\w})\cong(\widetilde{F},\widetilde{\w_F})\times (\C^2,\w_0),$ where $(\widetilde{M},\widetilde{\w})$ and $(\widetilde{F},\widetilde{\w_F})$ are the universal covers of $(M,\w)$ and $(F,\w_F)$ respectively, $\w_0$ is the flat metric. Since $\dim F=1$, the holomorphic sectional curvature of $\w_F$ is either identically zero or quasi-positive. Hence, either $M$ is flat, or $F$ is $\CP^1$. When $F$ is $\CP^1$, $M$ is a $\CP^1$-bundle over a torus $\T^2$, and the universal cover is $\CP^1\times \C^2$.
	\end{proof}
	
	\section{Proof of Theorem \ref{t3}}\label{rbc}
	In this section, we provide a proof of Theorem \ref{t3}.
	\begin{proof}[\bf Proof of Theorem \ref{t3}]
		Denote $n=\dim M$ and $r=\rank E$. For any $t\in \N_+,\ m\in [0,k]\bigcap \N$ and $I_k=(i_1,\cdots,i_k)\in \N_+^k$ for some $k\in\N$,  we define
		$$I_{k}(m,t)=(i_1,\cdots,i_{m-1},t,i_m,\cdots,i_{k})\in \N_+^{k+1}.$$
		We equip $E$ with a Hermitian metric $H$ and denote its chern curvature tensor by $R^H$. Set
		\begin{equation}
			\delta_2=\min\limits_{x\in M}\min_{u\in T^{1,0}_xM\setminus0,V\in E_x\setminus0}-\frac{R^H(u,\bar{u},V,\bar{V})}{\vert u\vert^2_h\vert V\vert^2_H}.
		\end{equation} Since $(M,h)$ has positive real bisectional curvature, there exists $\delta_1>0$ such that for any $x\in M$, any unitary basis $\{\frac{\p}{\p z^1},\cdots,\frac{\p}{\p z^n}\}$ of $T_x^{1.0}M$ and any nonnegative constants $a=\{a_1,\cdots,a_n\}$ with $\vert a\vert^2=a_1^2+\cdots+a_n^2>0$, we have that
		\begin{equation}
			\sum\limits_{i,j=1}^n R_{i\bar{i}j\bar{j}}a_ia_j>\delta_1\vert a\vert^2.
		\end{equation}
		For any $x\in M$, we choose local holomorphic coordinates $(z^1,\cdots,z^n)$ centered at $x$ such that $h_{i\bar{j}}$ is the identity at $x$. Let $\psi$ be any holomorphic section of $H^0(M,((T^{1,0}M)^*)^{\otimes p}\otimes E^{\otimes l})$. Locally $$\psi=\sum\limits_{I_p,J_l}a_{I_p,J_l}dz^{i_1}\otimes\cdots \otimes dz^{i_p}\otimes e^{j_1}\otimes\cdots\otimes e^{j_l}$$ with $I_p=(i_1,\cdots,i_p)\in \N_+^p,J_l=(j_1,\cdots,j_l)\in\N_+^l$, where $\{e^k\}_{k=1}^r$ is a local holomorphic basis of $E$ such that $(H_{t\bar{s}})$ is diagonal at $x$. By a direct computation, we have the following Bochner type formula,
		\begin{align*}
			\p_u\p_{\bar{u}}\vert \psi\vert^2
			=&\langle D'_u\psi,\overline{D'_u\psi}\rangle+\sum\limits_{i,j=1}^nR^h_{u\bar{u}i\bar{j}}\sum\limits_{m=1}^p\sum\limits_{I_{p-1},J_l}a_{I_{p-1}(m,i),J_l}\overline{a_{I_{p-1}(m,j),J_l}}\\
			&-\sum\limits_{t,s=1}^rR^H_{u\bar{u}t\bar{s}}\sum\limits_{k=1}^l\sum\limits_{I_p,J_{l-1}}a_{I_p,J_{l-1}(k,t)}\overline{a_{I_p,J_{l-1}(k,s)}}
		\end{align*}
		holds for all $u\in T_x^{1,0}M$. Define a real nonnegative $(1,1)$-form $\alpha=\im\alpha_{i\bar{j}}dz^i\wedge d\bar{z}^j$ by
		\begin{equation}
			\alpha_{i\bar{j}}=\sum\limits_{m=1}^p\sum\limits_{I_{p-1},J_l}a_{I_{p-1}(m,i),J_l}\overline{a_{I_{p-1}(m,j),J_l}}=\sum\limits_{m=1}^p\langle \iota^m_{\frac{\p}{\p z^i}}(\psi),\overline{\iota^m_{\frac{\p}{\p z^j}}(\psi)}\rangle
		\end{equation}
		with $\iota^m_{\frac{\p}{\p z^i}}(\psi)=\psi(\cdots,\frac{\p}{\p z^i},\cdots;\cdots)$ where the contractions are done at the $m$-th slot. In a similar way, define a section $\beta=\beta_{t\bar{s}}e^t\otimes \overline{e^s}$ of $E\otimes \overline{E}$ such that the components $(\beta_{t\bar{s}})$ is a nonnegative Hermitian matrix by
		\begin{equation}
			\beta_{t\bar{s}}=\sum\limits_{k=1}^l\sum\limits_{I_p,J_{l-1}}a_{I_p,J_{l-1}(k,t)}\overline{a_{I_p,J_{l-1}(k,s)}}=\sum\limits_{k=1}^l\langle \iota^{p+k}_{(e^t)^*}(\psi),\overline{\iota^{p+k}_{(e^s)^*}(\psi)}\rangle
		\end{equation}
		with $\iota^{p+k}_{(e^t)^*}(\psi)=\psi(\cdots;\cdots,(e^t)^*,\cdots)$ where the contractions are done at the $(p+k)$-th slot. Now we can reselect the local holomorphic coordinates of $M$ and the local holomorphic basis of $E$ such that $(h_{i\bar{j}}),(H_{t\bar{s}})$ are still the identity and $(\alpha_{i\bar{j}}),(\beta_{t\bar{s}})$ are diagonal at $x$, then
		$
		\sum\limits_{i=1}^n\alpha_{i\bar{i}}=p\vert\psi\vert^2
		$
		and
		$
		\sum\limits_{t=1}^r\beta_{t\bar{t}}=l\vert\psi\vert^2.
		$
		So, for any holomorphic section $\psi\in H^0(M,((T^{1,0}M)^*)^{\otimes p}\otimes E^{\otimes l})$,
		\begin{align*}
			\langle \im \pp\vert \psi\vert^2,\alpha\rangle_h=&\sum\limits_{i=1}^n\p_i\p_{\bar{i}}\vert\psi\vert^2\alpha_{i\bar{i}}\\
			\geq& \sum\limits_{i,j=1}^nR^h_{i\bar{i}j\bar{j}}\alpha_{i\bar{i}}\alpha_{j\bar{j}}-\sum\limits_{i=1}^n\sum\limits_{t=1}^r R^H_{i\bar{i}t\bar{t}}\alpha_{i\bar{i}}\beta_{t\bar{t}}\\
			\geq& \delta_1\sum\limits_{i=1}^n\vert\alpha_{i\bar{i}}\vert^2+\delta_2\sum\limits_{i=1}^n\alpha_{i\bar{i}}\sum\limits_{t=1}^r\beta_{t\bar{t}}\\
			\geq&\delta_1\frac{(\sum\limits_{i=1}^n\alpha_{i\bar{i}})^2}{n}+\delta_2(\sum\limits_{i=1}^n\alpha_{i\bar{i}})(\sum\limits_{t=1}^r\beta_{t\bar{t}})\\
			=&(p+\frac{n\delta_2}{\delta_1}l)\frac{p\delta_1\vert\psi\vert^4}{n}
		\end{align*}
		holds for any $x\in M$. Let $C_E=\max\{1,-\frac{n\delta_2}{\delta_1}+2\}$. We assume that $\vert \psi \vert^2$ achieves its maximum at some point $x_0\in M$. If $p,l\in\N_+$ with $p\geq C_E l$ and $\psi\neq0$, then we have
		$$0\geq\langle \im \pp\vert \psi\vert^2,\alpha\rangle_h\geq (p+\frac{n\delta_2}{\delta_1})\vert\psi\vert^4\geq2l\vert\psi\vert^4>0$$
		at $x_0$, which is a contradiction. Hence
		\begin{equation}\label{vanish}
			H^0(M,((T^{1,0}M)^*)^{\otimes p}\otimes E^{\otimes l})=0\ \text{ for any } p\geq C_E l.
		\end{equation}
		The vanishing theorem \eqref{vanish} implies that $H^{2,0}_{\bp}(M)=0$. In particular, if $M$ is K\"{a}hlerian, $M$ is projective by Kodaira's theorem (\cite{Kod}). Then \eqref{vanish} and Campana-Peternell-Demailly's criterion (Lemma \ref{cdp}) imply that $M$ is rationally connected. The proof is completed.
	\end{proof}
	
	\section*{Appendix. Proof of $L^2$-vanishing Theorem}
	\begin{proof}[\bf Proof of Proposition \ref{L2}]
		Since $\w$ is complete, there exists an exhaustive sequence $(K_v)_{v\in\N}$ of compact subsets of $M$ and functions $\psi_v\in C^\infty(M,\R)$ (see. e.g. \cite[Section 8.2]{demaillycomplex}) such that
		\begin{itemize}
			\item $\psi_v=1$ in a neighborhood of $K_v$, $\supp\psi_v\subset K^\circ_{v+1}$,
			\item $0\leq\psi_v\leq 1$ and $|d\psi_v|^2_g\leq 2^{-v} \psi_v$ on $M$.
		\end{itemize}
		Let $\eta$ be a nonzero $L^2$-holomorphic $(p,0)$-form, then $|\eta|\in W^{1,\infty}_{loc}(M)$ and $\Sigma:=\{x\in M: \eta(x)=0\}$ is a proper subvariety of $M$, from Lemma \ref{L1} and Bochner-type formula \eqref{Bochner}, we have
		\begin{equation}\label{ibp2}
			\begin{split}
				&-\int_M \langle \p\psi_v,\bp|\eta|_g^2\rangle_g\cdot\w^n=\int_M\psi_v\Delta|\eta|_g^2\cdot\w^n\\
				\geq&\int_M \psi_v\langle \im\langle D'\eta,\overline{D'\eta} \rangle,\frac{\beta}{|\eta|_g^2}\rangle\cdot\w^n+\int_M\psi_v\cdot\kappa|\eta|_g^2\cdot\w^n+\frac{n!(\im)^{p^2}}{(n-p-1)!}\int_M\psi_v\im\pp|\eta|_g^2\wedge\frac{\eta\wedge\bar{\eta}}{|\eta|_g^2}\wedge\w^{n-p-1}\\
				\geq&\int_M \psi_v\langle \im\langle D'\eta,\overline{D'\eta} \rangle,\frac{\beta}{|\eta|_g^2}\rangle\cdot\w^n+\frac{n!(\im)^{p^2}}{(n-p-1)!}\int_M\psi_v\im\pp|\eta|_g^2\wedge\frac{\eta\wedge\bar{\eta}}{|\eta|_g^2}\wedge\w^{n-p-1}
			\end{split}
		\end{equation}
		where $\beta=\Lambda^{p-1}(\frac{\eta\wedge\overline{\eta}}{p!})$. Recall \eqref{tr1} that $\tr_\w\beta=|\eta_g|^2$, then $|\frac{\beta}{|\eta|_g^2}|_g\leq 1$ and $|\frac{\eta\wedge\bar{\eta}}{|\eta|_g^2}|_g\leq 1$ on $M\setminus \Sigma$, thus each integral on the RHS of \eqref{ibp2} makes senses. For any fixed point $x_0\in M$, we choose local holomorphic coordinates $(z^1,\cdots,z^n)$ centered at $x_0$ such that $(g_{i\bar{j}})$ is the identity and $(\beta_{i\bar{j}})$ is diagonal, then
		\begin{equation}
			4\langle\im\p|\eta|_g\wedge\bp|\eta|_g,\beta\rangle_g=\frac{\p_i|\eta|_g^2\cdot{\p}_{\bar{i}}|\eta|_g^2}{|\eta|_g^2}\beta_{i\bar{i}}=\frac{\langle D'_i\eta,\bar{\eta}\rangle\cdot\langle\eta,\overline{D'_i\eta}\rangle}{|\eta|_g^2}\beta_{i\bar{i}}\leq \langle \im\langle D'\eta,\overline{D'\eta} \rangle,\beta\rangle.
		\end{equation}
		As the $L^2$-holomorphic $(p,0)$-form $\eta$ is closed (see \cite{gromov91}), applying Lebesgue's dominated covergence theorem, we have
		\begin{equation}
			\begin{split}
				0=&\int_M\p(\psi_v\bp|\eta|_g^2\wedge\frac{\eta\wedge\bar{\eta}}{|\eta|_g^2}\wedge\w^{n-p-1})\\
				=&\lim_{\epsilon\rightarrow0}\int_M\p(\psi_v\bp|\eta|_g^2\wedge\frac{\eta\wedge\bar{\eta}}{|\eta|_g^2+\epsilon}\wedge\w^{n-p-1})\\
				=&\lim_{\epsilon\rightarrow0}\int_M\left(\p\psi_v\wedge\bp|\eta|_g^2\wedge\frac{\eta\wedge\bar{\eta}}{|\eta|_g^2+\epsilon}+\psi_v\pp|\eta|_g^2\wedge\frac{\eta\wedge\bar{\eta}}{|\eta|_g^2+\epsilon}-4\psi_v\p|\eta|_g\wedge\bp|\eta|_g\wedge\frac{|\eta|_g^2}{(|\eta|_g^2+\epsilon)^2}\eta\wedge\bar{\eta}\right)\wedge\w^{n-p-1}\\
				=&\int_M\left(\p\psi_v\wedge\bp|\eta|_g^2\wedge\frac{\eta\wedge\bar{\eta}}{|\eta|_g^2}+\psi_v\pp|\eta|_g^2\wedge\frac{\eta\wedge\bar{\eta}}{|\eta|_g^2}-4\psi_v\p|\eta|_g\wedge\bp|\eta|_g\wedge\frac{\eta\wedge\bar{\eta}}{|\eta|_g^2}\right)\wedge\w^{n-p-1}.
			\end{split}
		\end{equation}
		Applying Lemma \ref{L1}, we have
		\[
		\begin{aligned}
			&\int_M \psi_v|\p|\eta|_g|_g^2\cdot\w^n
			=\int_M\psi_v\tr_\w(\im\p|\eta|_g\wedge\bp|\eta|_g)\cdot\w^n\\
			=&\int_M\psi_v\langle\im\p|\eta|_g\wedge\bp|\eta|_g,\frac{\beta}{|\eta|_g^2}\rangle_g\cdot\w^n+\frac{(\im)^{p^2}n!}{(n-p-1)!}\int_M\psi_v\im\p|\eta|_g\wedge\bp|\eta|_g\wedge\frac{\eta\wedge\bar{\eta}}{|\eta|_g^2}\wedge\w^{n-p-1}\\
			\leq&\frac{1}{4}\left(\int_M \psi_v\langle \im\langle D'\eta,\overline{D'\eta} \rangle,\frac{\beta}{|\eta|_g^2}\rangle\cdot\w^n+\frac{(\im)^{p^2}n!}{(n-p-1)!}\lim_{\epsilon\rightarrow0}\int_M\psi_v\im\p|\eta|_g\wedge\bp|\eta|_g\wedge\frac{\eta\wedge\bar{\eta}}{|\eta|_g^2+\epsilon}\wedge\w^{n-p-1}\right)\\
			=&\frac{1}{4}\left(\int_M \psi_v\langle \im\langle D'\eta,\overline{D'\eta} \rangle,\frac{\beta}{|\eta|_g^2}\rangle\cdot\w^n+\frac{n!(\im)^{p^2}}{(n-p-1)!}\int_M\psi_v\im\pp|\eta|_g^2\wedge\frac{\eta\wedge\bar{\eta}}{|\eta|_g^2}\wedge\w^{n-p-1}\right)\\
			&-\frac{(\im)^{p^2}n!}{4(n-p-1)!}\int_M\im\p\psi_v\wedge\bp|\eta|^2_g\wedge\frac{\eta\wedge\bar{\eta}}{|\eta|_g^2}\wedge\w^{n-p-1}\\
			\leq &C_1\int_M|\p\psi_v|_g\cdot\big\vert\p|\eta|_g\big\vert\cdot|\eta|_g\cdot\w^n
		\end{aligned}
		\]
		for some constant $C_1>0$ independent of $v\in\N$. The fact that $|\p\psi_v|_g^2\leq 2^{-v}\psi_v$ and the Schwarz inequality yields
		$$\int_M\psi_v\big|\p|\eta|_g\big|_g^2\cdot\w^n\leq 2^{-v}C_1||\eta||_{L^2},$$
		hence $\int_M\psi_v\big|\p|\eta|_g\big|_g^2\cdot\w^n\rightarrow0$ for $v\rightarrow\infty$, which implies that $|\eta|_g\equiv C$ for some constant $C\geq0$. Applying the pointwise computation in the proof of Theorem \ref{t2}, the fact \eqref{para} tells us that $D\eta=0$. The proof is completed.
	\end{proof}

	\bigskip
	
	\normalem
	\bibliographystyle{plain}

\end{document}